\documentclass[10pt]{amsart}
\usepackage{amsmath,amssymb,amsthm,amscd}
\usepackage{amsmath}
\usepackage{amssymb}
\usepackage{amssymb,amsmath,amsthm,amsfonts,latexsym,euscript}
\usepackage{wasysym}
\usepackage{fullpage}
\usepackage{hyperref}
\newtheorem*{definition}{Definition}
\newtheorem*{lemma}{Lemma}

\newtheorem*{theorem}{Theorem}

\usepackage{amssymb}
\usepackage[T1]{fontenc}
\usepackage[latin1]{inputenc}
\usepackage[usenames,dvipsnames]{xcolor}
\usepackage{graphicx}
\usepackage{setspace}
\usepackage{txfonts}
\usepackage{longtable}
\usepackage{blindtext}
\usepackage{ulem}
\begin{document}
\title{The Sum of Four Squares Over Real Quadratic Number Fields}
\author{Katherine Thompson}
\maketitle
This paper concerns the representation of totally positive integers in real quadratic number fields as the sum of four squares. First we show that the sum of four squares is only universal over $\mathbb Q(\sqrt{5})$, providing a proof alternative to G\"otzky on a formula for the exact number of such representations. The remainder of this work provides general upper and lower bounds on the Eisenstein coefficients of the associated theta series. 
Concretely, we can show that in all but two number fields there is a nontrivial cusp form in the decomposition of the theta series. Multiple concrete examples of the theta series decomposition are provided; specifically, we provide an {exact formula for the number of representations of a locally represented totally positive integer as a sum of four squares over $\mathbb Q(\sqrt{2})$}. We also provide the theta series decomposition for the sum of four squares over $\mathbb Q(\sqrt{3})$, $\mathbb Q (\sqrt{13})$ and $\mathbb Q (\sqrt{17})$.

\section{Introduction}
Throughout, let $K$ be a number field ($K= \mathbb Q$ or $K$ real quadratic) with $\mathcal O_K$ its ring of integers. An \textbf{n-ary quadratic form over $\mathcal O_K$} is a form $$Q(\vec{x})= Q(x_1,...,x_n) = \displaystyle\sum_{1 \leq i \leq j \leq n} a_{ij}x_ix_j \in \mathcal O_K [x_1,...,x_n].$$
\begin{definition}
Let $Q: \mathcal O_K^n \to \mathcal O_K$ be a quadratic form over the ring of integers of a totally real number field $K$. Suppose $[K:\mathbb Q] = 2$ and let $\sigma_1, \sigma_2$ denote the distinct embeddings $K \hookrightarrow \mathbb R$. We say that $Q$ is \textbf{(totally) positive definite} if the following conditions hold:
\begin{itemize}
\item[1.] $Q(\vec{x})=0$ if and only if $\vec{x}=0$.
\item[2.] For all $\vec{x} \neq 0$, $\sigma_i(Q(\vec{x})) >0$ for $i \in \{1,2\}$.
\end{itemize}
\end{definition}
We denote by $\mathcal O_K^{+}$ the totally positive elements of $\mathcal O_K$.
\begin{definition}
Let $Q(\vec{x})$ be an $n$-ary positive definite quadratic form over $\mathcal O_K$ and let $m \in \mathcal O_K^{+}$. We say that $Q$ \textbf{represents} $m$ if there exist $\vec{x} \in \mathcal O_K^n$ such that $Q(\vec{x}) = m$. Moreover, for fixed $Q$ and $m$ we define the \textbf{representation number of $m$ by $Q$} as $$r_Q(m) := \{ \vec{x} \in \mathcal O_K^n \vert Q(\vec{x}) = m \}.$$
Note that when $Q$ represents $m$, $r_Q(m) >0$.
\end{definition}
The question of which elements are represented by a fixed positive definite quadratic form has a deep and long history. Perhaps the first historically valued result is the following:
\begin{theorem}[Lagrange \cite{Lagrange}]
(1770) Every positive integer can be expressed as the sum of four squares.
\end{theorem}
This example was crucial in that the form represented all positive integers. That is, it was \textbf{universal}. Classifying the universal positive definite quadratic forms over $\mathbb Z$ was a long standing problem. An enumeration of results towards the final answer:
\begin{theorem}
\begin{enumerate}
\item There is no universal positive definite quadratic form in three variables over $\mathbb Z$.
\item (Ramanujan-Dickson) There are precisely $54$ diagonal universal quaternary forms.
\item (Halmos) A diagonal quaternary form is universal iff it represents $1$ through $15$.
\item The 15 Theorem: (Conway-Schneeberger, Bhargava) A classical form is universal iff it represents $1$ through $15$. Moreover, up to equivalence there are $204$ such quaternary forms.
\item The 290 Theorem: (Bhargava Hanke) A form is universal iff it represents $1$ through $290$. Moreover, up to equivalence, there are $6436$ such quaternary forms.
\end{enumerate}
\end{theorem}
Proving that numbers are represented by a fixed quadratic form, however, is different than providing the exact representation numbers. Results here include:
\begin{theorem}[Jacobi]
(1834) Let $Q$ be the sum of four squares over $\mathbb Z$ and let $m \in \mathbb N$. Then $$r_Q(m) = 8 \cdot \displaystyle\sum_{0<d \vert m, 4 \nmid d}d.$$
\end{theorem}
\begin{theorem}[G\"otzsky, Thompson]
The sum of four squares over $\mathbb Q(\sqrt{5})$ is universal. Moreover, for $m \in \mathcal O_{\mathbb Q (\sqrt{5})}^+$ $$r_Q(m) =8 \displaystyle\sum_{0 \neq (d) \vert m}N_{K/ \mathbb Q}(d) - 4 \displaystyle\sum_{2 \vert (d) \vert m}N_{K/ \mathbb Q}(d) + 8 \displaystyle\sum_{4 \vert (d) \vert m}N_{K/ \mathbb Q}(d). $$
\end{theorem}
The proofs of these theorems are analytic, relying heavily on the theory of modular forms and in particular the theory of {local densities} developed by Siegel.  Throughout let $Q$ be a quaternary positive-definite integral quadratic form over a real quadratic number field $K$ (or $K= \mathbb Q$). [Note that all definitions can be altered so that $Q$ has three or more variables and is defined over any totally real number field. Note also that we refer to the symmetric matrix associated to $Q$ as $M_Q$.] We define the \textbf{level} ${\mathfrak N}_Q$ of $Q$ to be the largest $\mathcal O_K$ ideal such that for all primes $\mathfrak{p}$, ${\mathfrak N}_{\mathfrak{p}}(2M_{Q,{\mathfrak{p}}})^{-1}$ is a matrix of integral ideals with diagonal entries in $2 \mathcal O_K$. We also define the \textbf{determinant} ${\mathfrak D}_Q$ of $Q$ to be $\det(M_Q)$. Last, we set the following Hecke character defined for all $\mathfrak{p} \nmid 2{\mathfrak N}_Q$ by $$\chi_Q(\mathfrak{p}) = \left( \dfrac{\mathfrak D_Q}{\mathfrak{p}} \right).$$ We then define the \textbf{theta series} associated to $Q$ as $$\Theta_Q(z) = 1 + \displaystyle\sum_{m \in ({\mathfrak{d}^{-1}})^+} r_Q(m) e^{2 \pi i Tr(m\cdot z)}.$$

\begin{theorem}
$\Theta_Q(m)$ is a Hilbert modular form of (parallel) weight $k=2$ over $\Gamma_0(\mathfrak N_Q)$ with associated character $\chi_Q$.
\end{theorem}
\begin{proof}
See \cite[Theorem 2.2, pg. 61]{Andrianov}.
\end{proof}
 As each space of modular forms of fixed weight, level and character decomposes into a direct sum of Eisenstein series and cusp forms, we write $\Theta_Q(z) = E_Q(z) \oplus S_Q(z)$, and $$r_Q(m) = a_E(m) + a_S(m)$$ for each coefficient in the Fourier expansion of $\Theta_Q$.

For $Q$ the sum of four squares the character $\mathfrak D_Q$ will always be trivial. Moreover, the level ideal $\mathfrak N_Q$ is $(4)$.
\subsection{Computing Local Densities}
\begin{theorem}[Siegel] Let $Q$ be a positive definite quadratic form over $\mathcal O_K$. Let $r_Q(m)=a_E(m)+a_C(m)$ be the decomposition of the Fourier coefficient of the $\Theta$ series into Eisenstein and cusp components. The $m^{th}$ Eisenstein coefficient of the theta series $\Theta_Q$ is given by
\begin{eqnarray*}
a_E(m) & = & \displaystyle\prod_{\mathfrak p} \beta_{\mathfrak p}(m) 
\end{eqnarray*}
where $\mathfrak p$ runs over all places of $K$ and where $\beta_{\mathfrak p}(m)$ is the $\mathfrak p$-adic local density, defined to be $$\beta_{\mathfrak p}(m) := \displaystyle\lim_{U \to \{m\}} \dfrac{\operatorname{Vol}(Q^{-1}(U))}{\operatorname{Vol}(U)}$$ where $U$ is an open neighborhood of $m$ in $K_{\mathfrak p}$ and where the volume is determined by a fixed Haar measure.
\end{theorem}
\noindent In practice, we consider
\begin{eqnarray*}
a_E(m) & = & \displaystyle\prod_{\mathfrak p} \beta_{\mathfrak p}(m) \\
& = & \left({ \displaystyle\prod_{\mathfrak p \vert \infty} \beta_{\mathfrak p}(m)}\right) {\left( \displaystyle\prod_{\mathfrak p \vert N_Q}\beta_{\mathfrak p}(m)\right) } {\left( \displaystyle\prod_{\mathfrak p \vert m, \mathfrak p \nmid N_Q} \beta_{\mathfrak p} (m) \right)} {\left( \displaystyle\prod_{\mathfrak p \nmid N_Qm} \beta_{\mathfrak p}(m)\right) }.
\end{eqnarray*}

\noindent For the product over the infinite places, a second result of Siegel applies (note that the reference provided states the most general case while what is stated is not):
\begin{theorem}
\cite[Hilfssatz 72]{Siegel} Let $K$ be a totally real number field with discriminant $D$ and with $h = [K: \mathbb Q]$. Let $m \in \mathcal O_K^+$ and let $\sigma$ be the discriminant of a four-variable positive definite quadratic form $Q$ defined over $\mathcal O_K$. Then
\begin{eqnarray*}
\displaystyle\prod_{\mathfrak{p} \vert \infty} \beta_{\mathfrak p}(m) & = & \pi^{2h} D^{-3/2} \left(N_{K/ \mathbb Q}( \sigma) \right) ^{-1/2} N_{K/ \mathbb Q} (m). 
\end{eqnarray*}
\end{theorem}
When $Q$ is the sum of four squares over a real quadratic number field, this formula yields
\begin{eqnarray*}
\displaystyle\prod_{\mathfrak{p} \vert \infty} \beta_{\mathfrak p}(m) & = & \pi^4 D^{-3/2}N_{K/ \mathbb Q}(m).
\end{eqnarray*}

\noindent Using the methods of Hanke \cite{Hanke}, which in turn uses work of Kitaoka \cite{Kitaoka} the infinite product can be simplified:
\begin{theorem}
Let $Q$ be a quaternary positive definite integral quadratic form (with level ideal $N_Q$ and with character $\chi_Q$) defined over the totally real number field $K$. Let $m \in \mathcal O_K^+$. Then
\begin{eqnarray*}
\displaystyle\prod_{\mathfrak p \nmid N_Q m}\beta_{\mathfrak p}(m) & = & L_{K}(2, \chi_Q)^{-1} \left( \displaystyle\prod_{\mathfrak p \vert N_Qm} \dfrac{(N_{K/ \mathbb Q}(\mathfrak p))^2}{(N_{K/ \mathbb Q}(\mathfrak p))^2 - \chi_Q(\mathfrak p)} \right).
\end{eqnarray*}
\end{theorem}
\noindent All together, this means that in practice:
\begin{eqnarray}
a_E(m) & = & \displaystyle\prod_{\mathfrak p} \beta_{\mathfrak p}(m) \notag \\
& = & \left({ \displaystyle\prod_{\mathfrak p \vert \infty} \beta_{\mathfrak p}(m)}\right) {\left( \displaystyle\prod_{\mathfrak p \vert N}\beta_{\mathfrak p}(m)\right) } {\left( \displaystyle\prod_{\mathfrak p \vert m, \mathfrak p \nmid N} \beta_{\mathfrak p} (m) \right)} {\left( \displaystyle\prod_{\mathfrak p \nmid Nm} \beta_{\mathfrak p}(m)\right) } \notag\\
& = & \dfrac{\pi^4 D^{-3/2}N_{K/ \mathbb Q}(m)}{L_K(2, \chi_Q)} \left( \displaystyle\prod_{\mathfrak p \vert N_Q m} \dfrac{\beta_{\mathfrak p}(m)(N_{K/ \mathbb Q}(\mathfrak p))^2}{(N_{K/ \mathbb Q}(\mathfrak p))^2 - \chi_Q(\mathfrak p)} \right).
\end{eqnarray}
To compute $\beta_{\mathfrak p}(m)$ for the finitely-many primes $\mathfrak p \vert N_Qm$, there are two distinct approaches. The first of these options is outlined in Hanke \cite{Hanke}; this approach uses Henselian lifting arguments to reduce the problem of finding $\beta_{\mathfrak p}(m)$ to computing $r_{{\mathfrak p}^k}(m)$ for an explicit power $k$. The advantage to this method is that it applies in any totally real number field. Consequently, it is the method we will use. The second technique was developed by Yang \cite{Yang}. This work relates local densities to specific integrals which are computed directly. There is one crucial limitation with respect to our purposes: it is only highlighted in the case of $\mathbb Z$-integral quadratic forms.\\
\\
As for our purposes we will need number field results, we recall some terminology from Hanke \cite{Hanke} relevant to later sections.\\
\begin{definition}
Let $R_{\mathfrak{p}^v}(m) := \{ \vec{x} \in (\mathcal O_K/ {\mathfrak{p}^v} \mathcal O_K)^4 : Q(\vec{x}) \equiv m \pmod{\mathfrak{p}^v} \}$. Note that by definition $\# R_{\mathfrak{p}^v}(m)= r_{\mathfrak{p}^v}(m)$. \\
\indent $\vec{x} \in R_{\mathfrak{p}^v}(m)$ is of \textbf{Zero type} if $\vec{x} \equiv \vec{0} \pmod{\mathfrak p}$ (in which case, we say $\vec{x} \in R_{\mathfrak{p}^v}^{\operatorname{Zero}}(m)$ with $\#R_{\mathfrak{p}^v}^{\operatorname{Zero}}(m) := r_{\mathfrak{p}^v}^{\operatorname{Zero}}(m) $), is of \textbf{Good type} if $\mathfrak{p}^{v_j}\vec{x_j} \not\equiv \vec{0} \pmod{\mathfrak p}$ for some $j$ (in which case, we say $\vec{x} \in R_{\mathfrak{p}^v}^{\operatorname{\operatorname{Good}}}(m)$ with $\#R_{\mathfrak{p}^v}^{\operatorname{Good}}(m) := r_{\mathfrak{p}^v}^{\operatorname{Good}}(m) $), and is of \textbf{Bad type} otherwise.
\end{definition}

 We note that as we are only considering the sum of four squares, we will only need to consider Zero and Good type solutions. As for the reduction maps:
\begin{theorem}
$$r_{\mathfrak{p}^{k+ \ell}}^{\operatorname{Good}}(m) = N(\mathfrak{p})^{3\ell}r_{\mathfrak{p}^k}^{\operatorname{Good}}(m)$$ for $k \geq 2 {\operatorname{ord}}_{\mathfrak{p}}(2)+1$.
\end{theorem}
\begin{proof}
See \cite[Lemma 3.2]{Hanke}.
\end{proof}
\begin{theorem}
The map 
\begin{eqnarray*}
\pi_Z: R_{\mathfrak{p}^k}^{\operatorname{Zero}}(m) & \to & R_{\mathfrak{p}^{k-2}} \left( \dfrac{m}{\pi_{\mathfrak{p}}^2} \right)\\
\vec{x} & \mapsto & \pi_{\mathfrak{p}}^{-1} \vec{x} \pmod{\mathfrak{p}^{k-2}}
\end{eqnarray*}
is a surjective map with multiplicity $N(\mathfrak{p})^4$. 
\end{theorem}
\begin{proof}
See \cite[pg. 359]{Hanke}.
\end{proof}

\subsection{$L$-functions over Number Fields}
In order to obtain exact values for $a_E(m)$, it becomes necessary to compute special values of $L$-functions over number fields. As in the case of the sum of four squares, this $L$-function is just the $\zeta$ function, we restrict our attention to that case.\\
\\
\noindent Let $\chi$ be a nontrivial primitive quadratic character of conductor $D$. Then $$\zeta_{\mathbb Q}(2) L_{\mathbb Q}(2, \chi) = \zeta_K(2)$$ where $K= \mathbb Q (\sqrt{\chi(-1)D})$ (so $K$ is imaginary if $\chi(-1)=-1$ and real if $\chi(-1)=1$) \cite[pg.10]{Bumpetal}.
\begin{theorem}
Let $K$ be a quadratic number field with discriminant $D$. Then $$\zeta_K(s) = \zeta_{\mathbb Q}(s) L_{\mathbb Q}(s, \chi_D).$$ where $\chi_D$ is the quadratic Dirichlet character of conductor $\vert D \vert$.

\end{theorem}
\begin{proof}
\begin{eqnarray*}
\zeta_K(s) & = & \displaystyle\prod_{\mathfrak{p}} \left( 1- \dfrac{1 }{N(\mathfrak{p})^s} \right)^{-1}\\
& = & \left( \displaystyle\prod_{p \textrm{ inert}} \left( 1- \dfrac{1}{p^{2s}} \right)^{-1} \right) \left( \displaystyle\prod_{p \textrm{ split}} \left( 1- \dfrac{1}{p^s} \right)^{-2} \right) \left( \displaystyle\prod_{p \textrm{ ramified}} \left( 1- \dfrac{1}{p} \right)^{-1} \right)\\
& = & \zeta_{\mathbb Q}(s) \left( \displaystyle\prod_{\chi_D(p)=-1} \left( 1- \dfrac{\chi_D(p)}{p^s} \right)^{-1} \right) \left( \displaystyle\prod_{\chi_D(p)=1} \left( 1- \dfrac{\chi_D(p)}{p^s} \right)^{-1} \right)\\
& = & \zeta_{\mathbb Q}(s) L_{\mathbb Q}(s, \chi_D).
\end{eqnarray*}
\end{proof}

\noindent Methods for computing special values of $L$-functions over $\mathbb Q$ are found in Iwasawa \cite{Iwasawa}. We highlight the main results here.
Let $\chi$ be a Dirichlet character of conductor $D$ and let $s>1$. Then: 
\begin{eqnarray}
L_{\mathbb Q}(s, \chi)&=& \dfrac{\tau(\chi)}{2i^\delta}\left(\dfrac{2 \pi}{D}  \right)^s \dfrac{L_{\mathbb Q}(1-s, \overline{\chi})}{\Gamma(s) \cos (\pi(s- \delta)/2)}.
\end{eqnarray}
where $\tau(\chi)$ denotes the traditional Gauss sum $\tau(\chi) = \sum_{a=1}^{D}\chi(a) e^{2 \pi i a/D}$, and where $$\delta = \begin{cases} 0, & \textrm{ if } \chi(-1)=1, \\ 1, & \textrm{ if } \chi(-1)=-1. \end{cases}$$

We then use the supplementary result:
\begin{theorem}
For a Dirichlet character $\chi$ of conductor $D$ and for any integer $n \geq 1$, 
\begin{center}
$L_{\mathbb Q}(1-n, \chi) = - \dfrac{B_{n, \chi}}{n}$, where $B_{n, \chi} = D^{n-1}\displaystyle\sum_{a=1}^{D} \chi(a)B_n \left( \dfrac{a-D}{D} \right)$ and where $B_n(\cdot)$ denotes the $n^{\text{th}}$ Bernoulli polynomial.
\end{center}
\end{theorem}

Note that in our applications $n=2$, so, we have: 
\begin{eqnarray*}
L_{\mathbb Q}(-1, \chi) &=& -\dfrac{D}{2} \left(\displaystyle\sum_{a=1}^{D}\chi(a) B_2 \left(\dfrac{a-D}{D} \right) \right)\\
\zeta_K(2) & = & \dfrac{\pi^4}{6D} \left( \displaystyle\sum_{a=1}^{D} \chi(a) e^{2 \pi i a/D} \right) \left(\displaystyle\sum_{a=1}^{D}\chi(a) B_2 \left(\dfrac{a-D}{D} \right) \right).
\end{eqnarray*}
\noindent Since $\displaystyle\sum_{a=1}^{D} \chi_D(a) e^{2 \pi i a/D} = \displaystyle\sum_{a=0}^{D} \chi_D(a) e^{2 \pi i a/D}$, if $\chi$ is a character of conductor $D$, then $$\left| \displaystyle\sum_{a=0}^{D} \chi(a) e^{2 \pi i a/D} \right|= \sqrt{D}.$$
\noindent This means for $Q$ the sum of four squares over a real quadratic number field of discriminant $D$, we have 
\begin{eqnarray}
a_E(m) & = & \dfrac{6N_{K/ \mathbb Q}(m)}{D}\left(\displaystyle\sum_{a=1}^{D}\chi(a) B_2 \left(\dfrac{a-D}{D} \right) \right)^{-1}  \left( \displaystyle\prod_{\mathfrak p \vert N_Q m} \dfrac{\beta_{\mathfrak p}(m)(N_{K/ \mathbb Q}(\mathfrak p))^2}{(N_{K/ \mathbb Q}(\mathfrak p))^2 - \chi_Q(\mathfrak p)} \right)\\
& = & 6D N_{K/ \mathbb Q}(m) \left( \displaystyle\sum_{a=1}^{D}\chi_D(a) a(a-D) \right)^{-1}   \left( \displaystyle\prod_{\mathfrak p \vert N_Q m} \dfrac{\beta_{\mathfrak p}(m)(N_{K/ \mathbb Q}(\mathfrak p))^2}{(N_{K/ \mathbb Q}(\mathfrak p))^2 - \chi_Q(\mathfrak p)} \right).
\end{eqnarray}
\begin{lemma}
Let $Q$ be the sum of four squares. For $\mathfrak{p} \neq (2)$, $\mathfrak{p} \vert (m)$ $$\dfrac{\beta_{\mathfrak{p}}(m)N_{K/ \mathbb Q}(\mathfrak{p})^2}{N_{K/ \mathbb Q}(\mathfrak{p})^2-1} = \displaystyle\sum_{i=0}^{\operatorname{ord}_{\mathfrak{p}}(m)}N_{K/ \mathbb Q}(\mathfrak{p})^{-i}.$$
\end{lemma}

\begin{proof}
Suppose $\operatorname{ord}_{\mathfrak{p}}(m)=2N$, $N \in \mathbb N$. Here both Good and Zero type solutions exist and
\begin{eqnarray*}
\beta_{\mathfrak{p}}(m) & = & \displaystyle\lim_{v \to \infty} \dfrac{r_{ \mathfrak{p}^v }^{\operatorname{Good}}(m)}{N_{K / \mathbb Q}(\mathfrak{p})^{3v}} + \displaystyle\lim_{v \to \infty} \dfrac{r_{ \mathfrak{p}^v }^{\operatorname{Zero}}(m)}{N_{K/ \mathbb Q}(\mathfrak{p})^{3v}}\\
& = & \dfrac{N_{K/ \mathbb Q}(\mathfrak{p})^3+N_{K/ \mathbb Q}(\mathfrak{p}) (N_{K/ \mathbb Q}(\mathfrak{p}) -1)-1}{N_{K/ \mathbb Q}(\mathfrak{p})^3} \\ & & +
\displaystyle\lim_{v \to \infty} \dfrac{1}{N_{K/ \mathbb Q}(\mathfrak{p})^{3v}} \left( \left( \displaystyle\sum_{i=1}^{N-1} N_{K/ \mathbb Q}(\mathfrak p)^{4i}r_{{\mathfrak p}^{v-2i}}^{\operatorname{Good}}(m/\mathfrak{p}^{2i}) \right) + N_{K/ \mathbb Q}(\mathfrak{p})^{4N} r_{\mathfrak{p}^{v-2N}}^{\operatorname{Good}}(m/\mathfrak{p}^{2N})\right)\\
& = & \left( \displaystyle\sum_{i=0}^{N-1} N_{K/ \mathbb Q}(\mathfrak{p})^{-2i} \right) \left( \dfrac{N_{K/ \mathbb Q}(\mathfrak{p})^3+N_{K/ \mathbb Q}(\mathfrak{p}) (N_{K/ \mathbb Q}(\mathfrak{p}) -1)-1}{N_{K/ \mathbb Q}(\mathfrak{p})^3}  \right) + N_{K/ \mathbb Q}(\mathfrak{p})^{-2N} \left( 1-\dfrac{1}{N_{K/ \mathbb Q}(\mathfrak{p})^2} \right).
\end{eqnarray*}
So for such primes
\begin{eqnarray*}
\dfrac{\beta_{\mathfrak{p}}(m) N_{K/ \mathbb Q}(\mathfrak{p})^2}{N_{K/ \mathbb Q}(\mathfrak{p})^2-1} & = & \left( \displaystyle\sum_{i=0}^{N-1}N_{K/ \mathbb Q}(\mathfrak{p})^{-2i} \right) \left(1+ \dfrac{1}{N_{K/ \mathbb Q}(\mathfrak{p})}  \right) + N_{K/ \mathbb Q}(\mathfrak{p})^{-2N}\\
& = & \displaystyle\sum_{i=0}^{2N} N_{K/ \mathbb Q}(\mathfrak{p})^{-i}.
\end{eqnarray*}

The proof for odd orders behaves identically.
\end{proof}
\noindent In conclusion, for the sum of four squares we have $$a_E(m) = 6D N_{K/ \mathbb Q}(m) \left( \displaystyle\sum_{a=1}^D \chi_D(a) a(a-D) \right)^{-1} \left( \displaystyle\prod_{\mathfrak p \vert (2)} \dfrac{\beta_{\mathfrak p}(m) N_{K/ \mathbb Q}(\mathfrak p)^2}{N_{K/ \mathbb Q}(\mathfrak p)^2-1} \right) \left( \displaystyle\prod_{\mathfrak p \nmid (2), \mathfrak p \vert m} \left( \displaystyle\sum_{i=0}^{ord_{\mathfrak p}(m)} N_{K/ \mathbb Q} (\mathfrak p)^{-i} \right) \right).$$As the remaining terms of $a_E(m)$ depend upon the particular number field, we proceed to the next section.



\section{Universality of the Sum of Four Squares}
\begin{theorem}
The sum of four squares is not universal over $\mathcal O_K$ for real quadratic $K \neq \mathbb Q(\sqrt{5})$.
\end{theorem}
The proof uses very simple arithmetic, so we omit it here. Specifically, one can show directly that the sum of four squares fails to represent $$m = \begin{cases} \left \lceil \sqrt{d} \right \rceil + \sqrt{d}, & \mathcal O_K = \mathbb Z [\sqrt{d}] \\ 
\left \lfloor \tfrac{1+\sqrt{d}}{2}\right \rfloor + \tfrac{1+\sqrt{d}}{2}, & \mathcal O_K = \mathbb Z \left[ \tfrac{1+\sqrt{d}}{2} \right]. 
\end{cases}$$ 
\noindent That the sum of four squares is indeed universal over $\mathbb Q(\sqrt{5})$ is less trivial.
\begin{theorem}
$Q(\vec{x})=x_1^2+x_2^2+x_3^2+x_4^2$ is universal over $\mathbb Q (\sqrt{5})$.
\end{theorem}
\begin{proof}
Considering the Fourier coefficients of the theta series $$r_Q(m) = a_E(m) + a_C(m),$$ one first proves that 
$$a_E(m) = \begin{cases} 8 \cdot \left( \displaystyle\sum_{(0) \neq (d) \vert (m)} N(d) \right),  & (2) \nmid (m) \\ 
8 \cdot 3 \cdot \left( 1+ 10 \displaystyle\sum_{i=0}^{N-1} 4^{2(N-i+1)+1}  \right) \cdot \left( \displaystyle\sum_{(0) \neq (d) \vert (m), (2) \nmid (d)} N(d) \right), & \operatorname{ord}_{(2)}(m) = 2N+1, N \geq 0\\
8 \cdot 3 \cdot \left( 9+10 \displaystyle\sum_{i=0}^{N-2}4^{2(N-i+1)} \right) \cdot \left( \displaystyle\sum_{(0) \neq (d) \vert (m), (2) \nmid (d)} N(d) \right), & \operatorname{ord}_{(2)}(m) = 2N, N \geq 1. \end{cases}$$
The next step is showing $r_Q(m) \equiv a_E(m)$ for all $m \in \mathcal O_{\mathbb Q(\sqrt{5})}^+$. It is known that the space of Hilbert modular forms of parallel weight two and level $4$ has a trivial cuspidal space. This shows $a_C(m) \equiv 0$ which would in turn imply $r_Q(m) \equiv a_E(m)$ (see \cite{Lassina}, Example 10). One can also verify this claim using code included in Appendix A.2 of \cite{Thompson} which implements Demb\'el\'e's algorithm over $\mathbb Q(\sqrt{5})$ and trivial character. That the cusp space is trivial completes the proof of universality. Finally, additional algebraic manipulations will yield the formula
$$r_Q(m) = 8 \displaystyle\sum_{0 \neq (d) \vert m}N(d) - 4 \displaystyle\sum_{2 \vert (d) \vert m}N(d) + 8 \displaystyle\sum_{4 \vert (d) \vert m}N(d).$$
\end{proof}



\section{Non-Universal Cases}
\subsection{$K = \mathbb Q (\sqrt{d})$, $d \equiv 2 \pmod{4}$}
$$a_E(m) = 6D N_{K/ \mathbb Q}(m) \left( \displaystyle\sum_{a=1}^D \chi_D(a) a(a-D) \right)^{-1} \left( \displaystyle\prod_{\mathfrak p \vert (2)} \dfrac{\beta_{\mathfrak p}(m) N_{K/ \mathbb Q}(\mathfrak p)^2}{N_{K/ \mathbb Q}(\mathfrak p_2)^2-1} \right) \left( \displaystyle\prod_{\mathfrak p \nmid (2), \mathfrak p \vert m} \left( \displaystyle\sum_{i=0}^{ord_{\mathfrak p}(m)} N_{K/ \mathbb Q} (\mathfrak p)^{-i} \right) \right)$$
Here $(2)=\mathfrak p_2^2$ ramifies, and so this simplifies to $$a_E(m) = 8D N_{K/ \mathbb Q}(m) \left( \displaystyle\sum_{a=1}^D \chi_D(a) a(a-D) \right)^{-1} \beta_{\mathfrak p_2}(m) \left( \displaystyle\prod_{\mathfrak p \nmid (2), \mathfrak p \vert m} \left( \displaystyle\sum_{i=0}^{ord_{\mathfrak p}(m)} N_{K/ \mathbb Q} (\mathfrak p)^{-i} \right) \right)$$
\begin{theorem}
Let $m \in \mathcal O_K^+$ be locally represented. The values which have obvious local obstructions at $\mathfrak p_2$ are those for which $ord_{\mathfrak p_2}(m)=1$. Let $R = \{ 4,2\sqrt{d}+2, 2\sqrt{d}-1, 2\sqrt{d}-3, 2 \sqrt{d}+3, 2 \sqrt{d}+1, -3,3,1,-1,2,-2,2\sqrt{d}-2,2\sqrt{d}, 2\sqrt{d}+4 \}$. Suppose that $m = (\mathfrak p_2)^{2k}r$, where $k \geq 0$ and where $r \pmod{\mathfrak p_2^5} \in R$. Then $$\beta_{\mathfrak p_2}(m) = \begin{cases} 2, & \mathfrak p_2 \nmid m, ord_{\mathfrak p_2}(m)=2 \\ \dfrac{9}{4} \left( \dfrac{2^{2(N-1)}-1  }{2^{2(N-1)} -2^{2(N-2)} } \right)+ \dfrac{3}{2^{2(N-1)+1 }}, & ord_{\mathfrak p_2}(m) = 2N+1, N \geq 1\\  ord_{\mathfrak p_2}(m) = 2N+1, N \geq 1 \\ \dfrac{9}{4} \left( \dfrac{2^{2(N-2)}-1}{2^{2(N-2)}-2^{2(N-3)}} \right) + \dfrac{7}{4} \dfrac{1}{2^{2(N-2)}} + \dfrac{1}{2^{2N-3}}, & ord_{\mathfrak p_2}(m)=2N, N \geq 2.\end{cases}$$ 
\end{theorem}
\begin{proof}
Using the notation of Hanke, all solutions will be of Good type. 
Hence we need only compute $r_{\mathfrak p_2^5}(m)$. We provide a three step process below for this. Note that $\left( \mathcal O_K / \mathfrak p_2 \mathcal O_K \right) \cong \mathbb F_2$. 
\begin{itemize}
\item[Step $1$.] $r_{\mathfrak p_2}(1) = 8$. This can be verified quickly using Sage \cite{SAGE}. Below is the collection of all such vectors $\vec{v} \in \mathbb F_2^4$ with $Q(\vec{v}) \equiv 1 \pmod{\mathfrak p_2}$:
\begin{center}
\begin{tabular}{cccc}
(0,0,0,1) & (0,0,1,0) & (0,1,0,0) & (1,0,0,0) \\
(1,1,1,0) & (1,1,0,1) & (1,0,1,1) & (0,1,1,1).
\end{tabular}
\end{center}
\item[Step $2$.] We separate these eight vectors into two families, based upon the rows above.
\item[Step $3$.] We now count the number of vectors $\vec{v} \in \left( \mathcal O_K / \mathfrak p_2^5 \mathcal O_K \right)$ such that $\vec{v}$ reduces to a vector above and such that $Q(\vec{v}) \equiv 1 \pmod{\mathfrak p_2^5}$. For a fixed vector in each family there are $2^{13}$ lifts. That makes the total $2\cdot 2^2 \cdot 2^{13} = 2^{16}$. Similar results occur for the remaining cases. 
\end{itemize}
Next suppose $ord_{\mathfrak p_2}(m)=2N+1, N \geq 1$ (Note $N =0$ is not possible, as such $m$ are not locally represented). Then there is the potential for both Zero type and Good type solutions, and
\begin{eqnarray*}
\beta_{\mathfrak p_2}(m) & = & \displaystyle\lim_{\nu \to \infty} \dfrac{r_{\mathfrak p_2^{\nu}}^G(m) }{2^{3\nu}} + \displaystyle\lim_{\nu \to \infty} \dfrac{r_{\mathfrak p_2^{\nu}}^G(m) }{2^{3\nu}}\\
& = & \dfrac{2^{13} \cdot 9}{2^{15}} + \displaystyle\lim_{\nu \to \infty} \left( \displaystyle\sum_{i=1}^{N-2} 2^{4i}2^{3(\nu-2i-5)}r_{\mathfrak p_2^5}^G(m/\mathfrak p_2^{2i})\right) + \displaystyle\lim_{\nu \to \infty} \dfrac{1}{2^{3v}} \left( 2^{4(N-1)}2^{3(v-2(N-1)-5)}r_{\mathfrak p_2^5}(m/\mathfrak p_2^{2N-2}) \right) \\
& = & \dfrac{9}{4} \left( \displaystyle\sum_{i=0}^{N-2} \dfrac{1}{2^{2i}} \right) + \dfrac{3}{2^{2(N-1)+1}}
\end{eqnarray*}
and the result follows.
Last suppose $ord_{\mathfrak p_2}(m)=2N$, $N \geq 2$. Then both Good and Zero type solutions exist with
\begin{eqnarray*}
\beta_{\mathfrak p_2}(m) & = & \displaystyle\lim_{\nu \to \infty} \dfrac{r_{\mathfrak p_2^{\nu}}^G(m) }{2^{3\nu}} + \displaystyle\lim_{\nu \to \infty} \dfrac{r_{\mathfrak p_2^{\nu}}^G(m) }{2^{3\nu}}\\
& = & \dfrac{2^{13}9}{2^{15}} + \displaystyle\lim_{\nu \to \infty} \dfrac{1}{2^{3\nu}} \left( \displaystyle\sum_{i=1}^{N-3}2^{4i}2^{3(\nu-2i-5)}r_{\mathfrak p_2^5}^G(m/ \mathfrak p_2^{2i}) \right)+  \displaystyle\lim_{\nu \to \infty} \dfrac{1}{2^{3\nu}}\left( 2^{4(N-2)}2^{3(\nu-2(N-2)-5)}r_{\mathfrak p_2^5}^G(m/ \mathfrak p_2^{2(N-2)}) \right)\\
& & +  \displaystyle\lim_{\nu \to \infty} \dfrac{1}{2^{3\nu}} \left( 2^{4(N-1)}2^{3(\nu-2(N-1)-5)}r_{\mathfrak p_2^5}(m/ \mathfrak p_2^{2(N-1)})  \right)\\
& = & \dfrac{9}{4} \left( \displaystyle\sum_{i=0}^{N-3} \dfrac{1}{2^{2i}} \right) + \dfrac{7}{4} \cdot \dfrac{1}{2^{2(N-2)}} + \dfrac{1}{2^{2N-3}}
\end{eqnarray*}
and the claim holds. 
\end{proof}

\subsubsection{Concrete example: $K= \mathbb Q (\sqrt{2})$}
We begin with $$a_E(m) = 4N_{K/ \mathbb Q}(m) \beta_{\mathfrak{p}_2}(m) \left( \displaystyle\prod_{\mathfrak p \nmid (2), \mathfrak p \vert m} \left(  \displaystyle\sum_{i=0}^{ord_{\mathfrak p}(m)} N_{K/ \mathbb Q}(\mathfrak p)^{-i} \right)  \right) .$$
Using Magma \cite{Magma}, we discover that the corresponding cusp space of Hilbert modular forms is $0$-dimensional. Therefore, for all locally represented $m \in \mathcal O_K^+$, $r_Q(m)=a_E(m)$. In particular, this means that an exact formula for $r_Q(m)$ can be provided and it is $$r_Q(m) = 8 \left( \displaystyle\sum_{0 \neq(d) \vert m}N(d) \right) - 6 \left( \displaystyle\sum_{(2) \vert (d) \vert m}N(d) \right) + 4 \left( \displaystyle\sum_{(4) \vert (d) \vert m}N(d)\right).$$

\subsection{$K= \mathbb Q (\sqrt{D})$, $D \equiv 1 \pmod{4}$}

Let $D \equiv 1 \pmod{4}$ be a squarefree positive integer. We have 
$$a_E(m) = 6D N_{K/ \mathbb Q}(m) \left( \displaystyle\sum_{a=1}^D \chi_D(a) a(a-D) \right)^{-1} \left( \displaystyle\prod_{\mathfrak p \vert (2)} \dfrac{\beta_{\mathfrak p}(m) N_{K/ \mathbb Q}(\mathfrak p)^2}{N_{K/ \mathbb Q}(\mathfrak p_2)^2-1} \right) \left( \displaystyle\prod_{\mathfrak p \nmid (2), \mathfrak p \vert m} \left( \displaystyle\sum_{i=0}^{ord_{\mathfrak p}(m)} N_{K/ \mathbb Q} (\mathfrak p)^{-i} \right) \right)$$
and as $\chi_D(-a)= \chi_D(a)$, we in fact have
$$a_E(m) = 3D N_{K/ \mathbb Q}(m) \left( \displaystyle\sum_{a=1}^{(D-1)/2} \chi_D(a) a(a-D) \right)^{-1} \left( \displaystyle\prod_{\mathfrak p \vert (2)} \dfrac{\beta_{\mathfrak p}(m) N_{K/ \mathbb Q}(\mathfrak p)^2}{N_{K/ \mathbb Q}(\mathfrak p_2)^2-1} \right) \left( \displaystyle\prod_{\mathfrak p \nmid (2), \mathfrak p \vert m} \left( \displaystyle\sum_{i=0}^{ord_{\mathfrak p}(m)} N_{K/ \mathbb Q} (\mathfrak p)^{-i} \right) \right).$$
Before proceeding to special cases, we mention a general bounding lemma:
\begin{lemma}
$ \left( \displaystyle\sum_{a=1}^{(D-1)/2} \chi_D(a) a(a-D)\right) \equiv 0 \pmod{2}$ and if $D >5$, $ \left( \displaystyle\sum_{a=1}^{(D-1)/2} \chi_D(a) a(a-D)\right) \equiv 0 \pmod{D}$. 
\end{lemma}
\begin{proof}
That the value is even is obvious. That it is divisible by $D$ for $D>5$ is not. First consider $D \equiv 1 \pmod{4}$ prime. Suppose $1 \neq q \in (\mathbb Z /D \mathbb Z)$ is a primitive root. Then


\begin{eqnarray*}
\chi_D(q) q^2 \displaystyle\sum_{a=0}^{D-1} \chi_D(a) a^2 & \equiv & \displaystyle\sum_{a=0}^{D-1}\chi_D(qa) (qa)^2 \pmod{D} \\
& \equiv & \displaystyle\sum_{b=0}^{D-1}\chi_D(b) b^2 \pmod{D}
\end{eqnarray*}
and
\begin{eqnarray*}
\underbrace{\left( \chi_D(q) q^2 - 1\right)}_{\ast} \displaystyle\sum_{a=0}^{D-1}\chi_D(a) a^2 & \equiv & 0 \pmod{D}.
\end{eqnarray*}
If $\ast \equiv 0 \pmod{D}$ then $\chi_D(q) q^2 \equiv 1 \pmod{D}$, and as $D>5$ then we have a contradiction. So since $\ast \not\equiv 0 \pmod{D}$, $2 \displaystyle\sum_{a=0}^{(D-1)/2} \chi_D(a) a^2  \equiv  \displaystyle\sum_{a=0}^{D-1} \chi_D(a) a^2 \equiv 0 \pmod{D}$ and the claim holds for $D$ prime.\\
\\
The composite case is quite similar. Instead of setting $q$ to be a primitive root, one selects $1 \neq q \in (\mathbb Z / D \mathbb Z)^{\times}$ such that for all $p \vert D$, $\chi_D(1) q^2 \not\equiv 1 \pmod{p}$. For any prime $p \neq 5$ that such a $q$ can be chosen is obvious. As for $p=5$, note that if $\chi_D(q) q^2 \equiv 1 \pmod{5}$ for all $q \in (\mathbb Z/ D \mathbb Z)^{\times}$ then $\chi_D(q)\equiv q^2 \pmod{5}$, which implies that $\chi_D(q)=\chi_5(q)$. As $D>5$, this is impossible.

\end{proof}

\subsubsection{$D \equiv 5 \pmod{8}$}
Here $(2)$ is inert, so
$$a_E(m) = 2D \cdot \dfrac{8}{5} \cdot N_{K/ \mathbb Q}(m) \left( \displaystyle\sum_{a=1}^{(D-1)/2} \chi_D(a) a(a-D) \right)^{-1} \beta_{\mathfrak (2)}(m) \left( \displaystyle\prod_{\mathfrak p \nmid (2), \mathfrak p \vert m} \left( \displaystyle\sum_{i=0}^{ord_{\mathfrak p}(m)} N_{K/ \mathbb Q} (\mathfrak p)^{-i} \right) \right)$$
\begin{lemma}
$$\beta_{(2)}(m) = \begin{cases}1, & (2) \nmid (m) \\ \dfrac{15}{8} \left(\displaystyle\sum_{i=0}^{N-1} \dfrac{1}{4^{2i}}  \right) + \dfrac{3}{4^{2N+1}}, &  \operatorname{ord}_{(2)}(m) = 2N+1, N \geq 0 \\ \dfrac{15}{8} \left( \displaystyle\sum_{i=0}^{N-2} \dfrac{1}{4^{2i}} \right)+\dfrac{27}{4^{2N}},  & \operatorname{ord}_{(2)}(m) =2N, N >0. \end{cases}$$
\end{lemma}
\begin{proof}
When $(2) \nmid (m)$, all solutions are of Good type and we have:
\begin{eqnarray*}
\beta_{(2)}(m) & = & 
\displaystyle\lim_{v \to \infty} \dfrac{r_{(2)^v}^{\operatorname{Good}}(m)}{4^{3v}} \\
& = &\displaystyle\lim_{v \to \infty} \dfrac{4^{3(v-3)}r_{(8)}^{\operatorname{Good}}(m)}{4^{3v}}\\
& = & \dfrac{262144}{2^{18}}=1.
\end{eqnarray*}
For $\operatorname{ord}_2(m)=2N+1$, $N \geq 0$, we have Good and Zero type solutions with
\begin{eqnarray*}
\beta_{(2)}(m) & = & \displaystyle\lim_{v \to \infty} \dfrac{r_{(2)^v}^{\operatorname{Good}}(m)}{4^{3v}} + \displaystyle\lim_{v \to \infty} \dfrac{r_{(2)^v}^{\operatorname{Zero}}(m)}{4^{3v}}\\
 & = & \displaystyle\lim_{v \to \infty }\dfrac{4^{3(v-3)}r_{(8)}^{\operatorname{Good}}(m) }{4^{3v}} + \displaystyle\lim_{v \to \infty} \dfrac{1}{4^{3v}}\left( \sum_{i=1}^{N}4^{4i} r_{(2)^{v-2i}}^{\operatorname{Good}} \left( \dfrac{m}{2^{2i}} \right) \right)\\
 & = & \dfrac{15}{8} + \dfrac{\sum_{i=1}^{N-1} 2^{15} \cdot 3 \cdot 5}{2^{18} \cdot 4^{2i}} + \dfrac{2^{16} \cdot 3}{4^9\cdot 4^{2N}} = \dfrac{15}{8} \left( \displaystyle\sum_{i=0}^{N-1} \dfrac{1}{4^{2i}} \right) + \dfrac{3}{4^{2N+1}}.
\end{eqnarray*}
Last, when $\operatorname{ord}_2(m)=2N$, $N \geq 1$, we have Good and Zero type solutions, and
\begin{eqnarray*}
\beta_{(2)}(m)  &=&  \displaystyle\lim_{v \to \infty} \dfrac{r_{(2)^v}^{\operatorname{Good}}(m)}{4^{3v}} + \displaystyle\lim_{v \to \infty} \dfrac{r_{(2)^v}^{\operatorname{Zero}}(m)}{4^{3v}}\\
& = & \dfrac{15}{8} + \dfrac{1}{4^9} \left( \sum_{i=1}^{N}4^{-2i}r_{(8)}^{\operatorname{Good}} \left( \dfrac{m}{2^{2i}} \right) \right)\\
& = & \dfrac{15}{8} + \displaystyle\sum_{i=1}^{N-2} \dfrac{r_{(8)}^{\operatorname{Good}} \left( \dfrac{m}{2^{2i}} \right)}{4^{2i+9}} + \dfrac{r_{(8)}^{\operatorname{Good}} \left( \dfrac{m}{2^{2(N-1)}} \right)}{4^{9+2(N-1)}} + \dfrac{r_{(8)}^{\operatorname{Good}} \left( \dfrac{m}{2^{2N}} \right)}{4^{2N+9}}\\
& = & \dfrac{15}{8} + \dfrac{15}{8} \left( \sum_{i=1}^{N-2} 4^{-2i} \right) + \dfrac{2^{15} \cdot 13}{2^{18}\cdot 4^{2(N-1)}} + \dfrac{2^{18}}{2^{18}\cdot 4^{2N}}
= \dfrac{15}{8} \left( \displaystyle\sum_{i=0}^{N-2} \dfrac{1}{4^{2i}} \right) + \dfrac{27}{4^{2N}}.
\end{eqnarray*}
\end{proof}

\begin{lemma}
$a_E(m) \geq \dfrac{192}{5D^{3/2}}N_{K/ \mathbb Q}(m) \beta_{(2)}(m) \left( \displaystyle\prod_{\mathfrak p \nmid (2), \mathfrak p \vert m} \left( \displaystyle\sum_{i=0}^{ord_{\mathfrak p}(m)} N(\mathfrak p)^{-i} \right)  \right)$.
\end{lemma}
\begin{proof}
This is based on a rather trivial bound that $L_{K}(2, \chi) \leq \dfrac{\pi^4}{36}$, and the fact that here $(2)$ is inert.
\end{proof}

In particular, this means that for $m$ odd we have: $$ \dfrac{192}{5D^{3/2}} \left(\displaystyle\sum_{0<d \vert m}N(d) \right) \leq a_E(m) \leq \dfrac{8}{5} \left(\displaystyle\sum_{0<d \vert m}N(d) \right) .$$

\subsubsection{Concrete Example:} $K = \mathbb Q (\sqrt{13})$. 
We see immediately that $$a_E(m) = \dfrac{8}{5} N_{K/ \mathbb Q}(m) \beta_{(2)}(m) \left( \displaystyle\prod_{\mathfrak p \nmid (2), \mathfrak p \vert m} \left(\displaystyle\sum_{i=0}^{ord_{\mathfrak p}(m)} N_{K/ \mathbb Q} (\mathfrak p)^{-i} \right) \right)$$ where the $(2)$-adic local density was classified above. Note that this gives $a_E(1)= \tfrac{8}{5}$, showing the constant in the upper bound given above is sharp. Using Magma \cite{Magma} we find that while the dimension of the space of Hilbert modular cusp forms of parallel weight $2$ and level $(4)$ is two, there are no newforms. Oldforms in the Hilbert setting are a very direct generalization of oldforms in the classical setting. Noting that the dimensions of the spaces of Hilbert modular cusp forms of level $(2)$ and $(1)$ respectively are one and zero, we see that the forms of level $(4)$ are generated by the one eigenform of level $(2)$ (which we call $f$) and by $f\vert V(2)$. Again, with the aid of Magma, one then can show that $$\Theta_Q(z) = E(z) + \dfrac{32}{5} f(z) + \dfrac{128}{5} f\vert V(2) (z).$$

\subsubsection{{$D \equiv 1 \pmod{8}$}}
Again, we begin with the local density at $(2)= \mathfrak p_1 \mathfrak p_2$. 
\begin{lemma}
$$\beta_{\mathfrak p_i}(m) = \begin{cases} 1, & \mathfrak p_i \nmid m \\ \dfrac{3}{2^{2N+1}}, & ord_{\mathfrak p_i}(m)=2N+1, N \geq 0 \\ \dfrac{3}{2^{2N}}, & ord_{\mathfrak p_i}(m) =2N, N \geq 1 \end{cases}$$
\end{lemma}
\begin{proof}
We note that $\left( \mathcal O_K/ \mathfrak p_i \mathcal O_K \right) \cong \mathbb F_2$. We begin by outlining a particular case of $\mathfrak p_i \nmid m$, noting that the other such cases behave identically. Then
\begin{itemize}
\item $r_{\mathfrak p_i}(1) = 2^3$. This again can be verified quickly using Sage \cite{SAGE} with data shown below:
\begin{center}
\begin{tabular}{cccc}
$(1,0,0,0)$ & $(0,1,0,0)$ & $(0,0,1,0)$ & $(0,0,0,1)$\\
$(1,1,1,0)$ & $(1,1,0,1)$ & $(1,0,1,1)$ & $(0,1,1,1)$
\end{tabular}
\end{center}
\item We again consider the perfect squares $\pmod{\mathfrak p_i^3}$ and their reductions $\pmod{\mathfrak p^i}$:
\begin{eqnarray*}
0 & \mapsto& 0, 4 \\
1 & \mapsto& 1
\end{eqnarray*}
We note again that each of the two classes has $4$ lifts, equally distributed (i.e., two elements square to $4$, and four elements square to $1$).
\item Again, we are reduced to a counting argument. We take into consideration the four vectors that lift:
\begin{center}
\begin{tabular}{cccc}
$(1,0,0,0)$ & $(0,1,0,0)$ & $(0,0,1,0)$ & $(0,0,0,1)$
\end{tabular}
\end{center}
Each vector has a total of $2^7$ lifts, bringing the grand total to $2^2 \cdot 2^7 = 2^9$.
\end{itemize}
For $ord_{\mathfrak p_i}(m) = 2N+1$, $N \geq 0$, we have Good and Zero type solutions with
\begin{eqnarray*}
\beta_{\mathfrak p_i}(m) & = & \displaystyle\lim_{\nu \to \infty} \dfrac{r_{\mathfrak p_i^\nu}^{Good}(m)}{2^{3 \nu}} + \displaystyle\lim_{\nu \to \infty} \dfrac{r_{\mathfrak p_i^\nu}^{Zero}(m)}{2^{3 \nu}}\\
& = & 0 + \displaystyle\lim_{\nu \to \infty} \dfrac{1}{2^{3 \nu}} \left( \displaystyle\sum_{i=1}^{N}2^{4i}r_{\mathfrak p_i^{v-2i}}^{Good}(m/\mathfrak p_i^2) \right)\\
& = & 2^{4N} \displaystyle\lim_{\nu \to \infty} \dfrac{2^{3(\nu-2N-3)}r_{\mathfrak p_i^3}(m/\mathfrak p_i^{2N}) }{2^{3\nu}}\\
& = & \dfrac{2^8 \cdot 3}{2^9 \cdot 2^{2N}} = \dfrac{3}{2^{2N+1}}.
\end{eqnarray*}
Last, suppose $ord_{\mathfrak p_i}(m)=2N$, $N \geq 1$. We have Good and Zero type solutions with
\begin{eqnarray*}
\beta_{\mathfrak p_i}(m) & = & \displaystyle\lim_{\nu \to \infty} \dfrac{r_{\mathfrak p_i^\nu}^{Good}(m)}{2^{3 \nu}} + \displaystyle\lim_{\nu \to \infty} \dfrac{r_{\mathfrak p_i^\nu}^{Zero}(m)}{2^{3 \nu}}\\
& = & 0 + \displaystyle\lim_{\nu \to \infty} \dfrac{1}{2^{3 \nu}} \left( \displaystyle\sum_{i=1}^{N}2^{4i}r_{\mathfrak p_i^{v-2i}}^{Good}(m/\mathfrak p_i^2) \right)\\
& = & \displaystyle\lim_{\nu \to \infty} \dfrac{2^{4(N-1)}2^{3(v-2(N-1)-3)}r_{\mathfrak p_i^{3}}^{Good}(m/2^{2(N-1)}) }{2^{3 \nu}} + \displaystyle\lim_{\nu \to \infty} \dfrac{2^{4N}2^{3(v-2N-3)}r_{\mathfrak p_i^{3}}^{Good}(m/2^{2N}) }{2^{3 \nu}}\\
& = & \dfrac{r_{\mathfrak p_i^3}^{Good}(m/2^{2(N-1)})}{2^92^{2(N-1)}} + \dfrac{r_{\mathfrak p_i^3}^{Good}(m/2^{2N})}{2^92^{2N}}\\
& = & \dfrac{2^8}{2^92^{2(N-1)}} + \dfrac{2^9}{2^92^{2N}} = \dfrac{3}{2^{2N}}.
\end{eqnarray*}
\end{proof}
Starting with 
\begin{eqnarray*}
a_E(m) & = & 3D N(m) \left( \displaystyle\sum_{a=1}^{(D-1)/2} \chi_D(a) a(a-D)\right)^{-1} \left( \displaystyle\prod_{\mathfrak p \vert (2)} \dfrac {\beta_{\mathfrak p}(m) N(\mathfrak p)^2}{N(\mathfrak p)^2-1} \right) \left( \displaystyle\prod_{\mathfrak p \vert m, \mathfrak \nmid (2)} \displaystyle\sum_{i=0}^{ord_{\mathfrak p}(m)} N(\mathfrak p)^{-i} \right) 
\end{eqnarray*}

\subsubsection{Concrete Example} $K= \mathbb Q (\sqrt{17})$. The corresponding space of Hilbert modular cusp forms is $5$-dimensional. There is one dimension of newforms at $(4)$ called $f_4$, one dimension at $(2)$ $f_2$, which gives the entire space, denoted by $f_{2,2}, f_{2, \mathfrak p_1}, f_{2, \mathfrak p_2}$. So note that $$\mathcal S_Q(z) = c_4 f_4(z)+c_2f_2(z)+c_{2,2}f_{2,2}(z)+c_{2, \mathfrak p_1}f_{2, \mathfrak p_1}(z)+ c_{2, \mathfrak p_2}f_{2, \mathfrak p_2}(z).$$
Here we have $$a_E(m) = \dfrac{4}{3} N(m) \left( \displaystyle\prod_{\mathfrak p \vert (2)} \beta_{\mathfrak p}(m) \right) \left(\displaystyle\prod_{\mathfrak p \nmid (2), \mathfrak p \vert m} \displaystyle\sum_{i=0}^{ord_{\mathfrak p}(m)} N(\mathfrak p)^{-i} \right).$$
We collect some data:
\begin{center}
\begin{tabular}{|c|c|c|}
\hline
$m$ & $r_Q(m)$ & $a_E(m)$ \\
\hline
$1$ & $8$ & $4/3$ \\
\hline
$2$ & $24$ & $12$ \\
\hline
$2+ (1+\sqrt{17})/2$ & $0$ & $4$\\
\hline
$11+7(1+\sqrt{17})/2$ & $0$ & $4$ \\
\hline
$3$ & $32$ & $40/3$ \\
\hline
$5$ & $48$ & $104/3$ \\
\hline
$6$ & $96$ & $120$\\
\hline
\end{tabular}
\end{center}
This immediately gives $c_2=4/3$ and $c_4=c_{2,2} = 16/3$, $c_{2, \mathfrak p_1} = c_{2, \mathfrak p_2} = -8/3$.
\subsection{$K= \mathbb Q (\sqrt{D})$, $D \equiv 3 \pmod{4}$}

Here $(2)=\mathfrak p_2^2$ ramifies, and so this simplifies to $$a_E(m) = 8D N_{K/ \mathbb Q}(m) \left( \displaystyle\sum_{a=1}^D \chi_D(a) a(a-D) \right)^{-1} \beta_{\mathfrak p_2}(m) \left( \displaystyle\prod_{\mathfrak p \nmid (2), \mathfrak p \vert m} \left( \displaystyle\sum_{i=0}^{ord_{\mathfrak p}(m)} N_{K/ \mathbb Q} (\mathfrak p)^{-i} \right) \right)$$
The locally represented squares mod $\mathfrak p_2^5$ are $0, -2d+1,4d,2d-1,2d,-2d,4d+2,2d+1,1,4d+1,-2d-1,2,4d-1,-1,-2d+2,2d+2$.
\begin{lemma}
Suppose that $m\in \mathcal O_K^{+}$ is locally represented. Then
$$\beta_{\mathfrak p_2} (m) = \begin{cases}2, & \mathfrak p_2 \nmid m, ord_{p_2}(m)=2 \\ \dfrac{9}{4} \left( \dfrac{2^{2(N-1)}-1  }{2^{2(N-1)} -2^{2(N-2)} } \right)+ \dfrac{3}{2^{2(N-1)+1 }}, & ord_{\mathfrak p_2}(m) = 2N+1, N \geq 1 \\ \dfrac{9}{4} \left( \dfrac{2^{2(N-2)}-1}{2^{2(N-2)}-2^{2(N-3)}} \right) + \dfrac{7}{4} \dfrac{1}{2^{2(N-2)}} + \dfrac{1}{2^{2N-3}}, & ord_{\mathfrak p_2}(m)=2N, N \geq 2. \end{cases}$$
\end{lemma}
\begin{proof}
As in the earlier proofs, when $m$ is odd, then all solutions are of Good type, and via Sage \cite{SAGE} one can determine that
\begin{eqnarray*}
\beta_{\mathfrak p_2}(m) & = & \displaystyle\lim_{\nu \to \infty} \dfrac{r_{\mathfrak p_2^\nu}^G(m)}{2^{3\nu}} \\
& = & \dfrac{r_{\mathfrak p_2^5}^G(m)}{2^{15}} = \dfrac{2^{16}}{2^{15}}.
\end{eqnarray*}
A similar argument holds for $ord_{\mathfrak p_2}(m)=2$. Note, however, that in this case, $m \equiv 2, 4\sqrt{d}+2, -2\sqrt{d}, 2 \sqrt{d}$. In these first two instances, there are only Good-type solutions. In the latter two cases, there are both Good type as well as Zero type solutions. \\
\\
Next suppose $ord_{\mathfrak p_2}(m)=2N+1, N \geq 1$ (Note $N =0$ is not possible, as such $m$ are not locally represented). Then there is the potential for both Zero type and Good type solutions, and
\begin{eqnarray*}
\beta_{\mathfrak p_2}(m) & = & \displaystyle\lim_{\nu \to \infty} \dfrac{r_{\mathfrak p_2^{\nu}}^G(m) }{2^{3\nu}} + \displaystyle\lim_{\nu \to \infty} \dfrac{r_{\mathfrak p_2^{\nu}}^G(m) }{2^{3\nu}}\\
& = & \dfrac{2^{13} \cdot 9}{2^{15}} + \displaystyle\lim_{\nu \to \infty} \left( \displaystyle\sum_{i=1}^{N-2} 2^{4i}2^{3(\nu-2i-5)}r_{\mathfrak p_2^5}^G(m/\mathfrak p_2^{2i})\right) + \displaystyle\lim_{\nu \to \infty} \dfrac{1}{2^{3v}} \left( 2^{4(N-1)}2^{3(v-2(N-1)-5)}r_{\mathfrak p_2^5}(m/\mathfrak p_2^{2N-2}) \right) \\
& = & \dfrac{9}{4} \left( \displaystyle\sum_{i=0}^{N-2} \dfrac{1}{2^{2i}} \right) + \dfrac{3}{2^{2(N-1)+1}}
\end{eqnarray*}
and the result follows.\\
Last suppose $ord_{\mathfrak p_2}(m)=2N$, $N \geq 2$. Then both Good and Zero type solutions exist with
\begin{eqnarray*}
\beta_{\mathfrak p_2}(m) & = & \displaystyle\lim_{\nu \to \infty} \dfrac{r_{\mathfrak p_2^{\nu}}^G(m) }{2^{3\nu}} + \displaystyle\lim_{\nu \to \infty} \dfrac{r_{\mathfrak p_2^{\nu}}^G(m) }{2^{3\nu}}\\
& = & \dfrac{2^{13}9}{2^{15}} + \displaystyle\lim_{\nu \to \infty} \dfrac{1}{2^{3\nu}} \left( \displaystyle\sum_{i=1}^{N-3}2^{4i}2^{3(\nu-2i-5)}r_{\mathfrak p_2^5}^G(m/ \mathfrak p_2^{2i}) \right)+  \displaystyle\lim_{\nu \to \infty} \dfrac{1}{2^{3\nu}}\left( 2^{4(N-2)}2^{3(\nu-2(N-2)-5)}r_{\mathfrak p_2^5}^G(m/ \mathfrak p_2^{2(N-2)}) \right)\\
& & +  \displaystyle\lim_{\nu \to \infty} \dfrac{1}{2^{3\nu}} \left( 2^{4(N-1)}2^{3(\nu-2(N-1)-5)}r_{\mathfrak p_2^5}(m/ \mathfrak p_2^{2(N-1)})  \right)\\
& = & \dfrac{9}{4} \left( \displaystyle\sum_{i=0}^{N-3} \dfrac{1}{2^{2i}} \right) + \dfrac{7}{4} \cdot \dfrac{1}{2^{2(N-2)}} + \dfrac{1}{2^{2N-3}}
\end{eqnarray*}
and the claim holds. 
\end{proof}

\subsubsection{Concrete Example:} $K = \mathbb Q (\sqrt{3})$ 
We have 
$$a_E(m) = 2N_{K/ \mathbb Q}(m) \beta_{\mathfrak p_2}(m) \left( \displaystyle\prod_{\mathfrak p \nmid (2), \mathfrak p \vert m} \left( \displaystyle\sum_{i=0}^{ord_{\mathfrak p}(m)} N_{K/ \mathbb Q} (\mathfrak p)^{-i} \right) \right).$$ The dimension of the corresponding space of Hilbert modular cusp forms is $1$ (with that form being a newform). As $a_E(1)=a_C(1)=4$, we immediately see that if $f$ is the eigenform we have $\Theta_Q(z) = E(z)+4f(z)$.
\section{Acknowledgements} The tools for this paper were developed in the course of the author's thesis at the University of Georgia. Given that, the author would like to thank her Ph.D. advisers Jon Hanke and Danny Krashen. We are also indebted to Jeremy Rouse and Joe Vandehey for comments and suggestions, and especially to the former in help obtaining software packages. Last, we thank Pete L. Clark and John Voight for reading early drafts of the manuscript. As we were preparing this manuscript for publication we learned of an undergraduate project by John Goes (under Ramin Takloo-Bighash at UIC in 2011) where some of these density computations were performed.

\end{document}